\documentclass{article} 
\usepackage{color}
\usepackage[utf8]{inputenc} 
\usepackage{setspace}
\usepackage{graphicx} 
\usepackage[parfill]{parskip}
\usepackage{titlesec}
\usepackage{changepage}
\usepackage{paralist} 
\usepackage{verbatim} 
\usepackage{mathrsfs}
\usepackage{amssymb}
\usepackage{amsmath}
\usepackage{titlefoot}
\usepackage{tikz}

\title{Construction of Regular Non-Atomic Strictly-Positive Measures in Second-Countable Non-Atomic Locally Compact Hausdorff Spaces}
\author{Jason Bentley\\
Department of Mathematics\\
University of Central Florida\\
Orlando, Florida, USA}

\newcommand{\mc}[1]{\mathcal{#1}}
\newcommand{\mb}[1]{\mathbb{#1}}

\newcommand{\eps}{\varepsilon}

\newtheorem{theorem}{Theorem}[section]
\newtheorem{lemma}[theorem]{Lemma}
\newtheorem{proposition}[theorem]{Proposition}

\newenvironment{proof}[1][Proof]{\begin{trivlist}
\item[\hskip \labelsep {\bfseries #1}]}{\end{trivlist}}
\newenvironment{definition}[1][Definition]{\begin{trivlist}
\item[\hskip \labelsep {\bfseries #1}]}{\end{trivlist}}
\newenvironment{example}[1][Example]{\begin{trivlist}
\item[\hskip \labelsep {\bfseries #1}]}{\end{trivlist}}

\newcommand{\qed}{\nobreak \ifvmode \relax \else
      \ifdim\lastskip<1.5em \hskip-\lastskip
      \hskip1.5em plus0em minus0.5em \fi \nobreak
      \vrule height0.75em width0.5em depth0.25em\fi}

\begin{document}
\onehalfspacing
\maketitle

\unmarkedfntext{\textit{Key words and phrases.} Regular measure, non-atomic measure, strictly-positive measure, locally compact spaces, non-atomic spaces, Polish spaces, regular spaces, second-countable spaces.}
\unmarkedfntext{2010 \textit{Mathematics Subject Classification.} Primary 28C15.}

\section*{Abstract}

This paper presents a constructive proof of the existence of a regular non-atomic strictly-positive measure on any second-countable non-atomic locally compact Hausdorff space. This construction involves a sequence of finitely-additive set functions defined recursively on an ascending sequence of rings of subsets with a premeasure limit that is extendable to a measure with the desired properties. Non-atomicity of the space provides a non-trivial way to ensure that the limit is a premeasure.

\pagebreak
\section{Introduction}

One well-known result from measure theory is the construction of the Lebesgue measure on the real line. One construction begins with a set function defined on the ring of finite unions of precompact intervals with rational endpoints which returns the total length. This construction of the Lebesgue measure has several crucial steps to show:

\begin{itemize}
\item[1. ] The set function has finite additivity on the ring. 
\item[2. ] The outer measure is at most $(\le)$ the set function on open sets in the ring. 
\item[3. ] The outer measure is at least $(\ge)$ the set function on compact sets in the ring. 
\item[4. ] The outer measure on boundaries of open sets in the ring (subsets of $\mb{Q}$) is zero. 
\item[5. ] The outer measure and the set function agree on the ring, which also shows that the set function is a pre-measure on the ring. 
\item[6. ] Apply Carath\'{e}odory's Extension Theorem to extend the set function to the Lebesgue measure on the Borel subsets of the real line.
\end{itemize} 

The Lebesgue measure on Borel sets is known to be regular, non-atomic and strictly-positive. On which topological spaces can measures with such properties be guaranteed?

This paper answers this question by applying a similar construction to 2nd-countable non-atomic locally compact Hausdorff spaces. The choice of such spaces is motivated by the steps involved in the Lebesgue measure construction. In fact, steps 2 and 3 are guaranteed with the usual choice of outer measure, and once step 4 is established, steps 5 and 6 immediately follow. 

Steps 1 and 4 prove to be challenging for these spaces. Indeed, the space is not necessarily a topological group, the ring of sets is only described topologically (or via a metric), and the finitely additive set function is not as clearly or easily defined. 

This paper provides a careful construction of a finitely-additive set function via the limit of a sequence of set functions defined recursively on a growing sequence of rings. The sequence of rings of sets and the sequence of set functions must be coupled together meticulously so that the outer measure satisfies step 4 and I thank Piotr Mikusi\'{n}ski for the comments which gave me that insight. Once the measure is formed, it can be shown to be regular, non-atomic and strictly-positive.

\section{Notation}

Let $(X,\tau_X)$ be a topological space and let $A \subseteq X$. Then $A^\circ$, $\overline{A}$, $\partial A$ and $A^e$ denote the interior, closure, boundary and exterior of $A$ respectively. 

For sets $A$ and $B$, the expression $A - B$ denotes the relative set complement: that is, $a \in A - B$ if $a \in A$ and $a \notin B$. In this paper, $``\subset"$ will always mean proper subset, and $``\subseteq"$ will always mean subset or equality. An open set $U$ is called \emph{regular} if $U$ is the interior of $\overline{U}$ and a closed set $F$ is called \emph{regular} if $F$ is the closure of $F^\circ$.

\section{Non-Atomic Topologies}

\begin{definition}
A topological space $(X,\tau_X)$, or a topology $\tau_X$ is \emph{non-atomic} if for all $x \in X$ and for every open $U$ containing $x$, there exists an open set $V$ with $x \in V \subset U$. 
\end{definition}

The following properties of non-atomic topological spaces can be easily verified by definition.

\begin{proposition}\label{prop_1}
~\newline
\vspace{-15 pt}
\begin{itemize}
\item[(a)] If $(X,\tau_X)$ is non-atomic, then every non-empty open set in $\tau_X$ must be infinite in cardinality.
\item[(b)] If $(X,\tau_X)$ is a $T_1$-space with all non-empty open sets having infinite cardinality, then $(X,\tau_X)$ is non-atomic.
\item[(c)] Let $(X,\tau_X)$ be a non-atomic topological space. Then every $x \in X$ and any open neighborhood $U$ of $x$ yield infinitely many open neighborhoods of $x$ which are proper subsets of $U$.
\end{itemize}
\end{proposition}

\begin{example}
Any normed space over the real or complex field is non-atomic. In particular, $(\mb{R},|\cdot|)$ is non-atomic since $x \in (a,b)$ implies $x \in (a+\eps, b-\eps) \subset (a,b)$ for any $\eps$ satisfying $0 < \eps < \min\{x-a, b-x\}$.
\end{example}

Non-atomic spaces have a nice property pertaining to topological bases.

\begin{proposition}\label{prop_2}
Let $(X,\tau_X)$ be a non-atomic topological space, let $\{U_i : i \in I\}$ be a topological basis for $\tau_X$, and let $J$ be a cofinite subset of $I$. Then $\{U_j : j \in J\}$ is also a topological basis for $\tau_X$.
\end{proposition}

For non-atomic locally compact Hausdorff spaces, the next proposition allows for open sets to be ``bored" by compact closures of open subsets. This will be utilized in Lemma \ref{lem_permutation} later.

\begin{proposition}\label{prop_3}
Let $(X,\tau_X)$ be a non-atomic locally compact Hausdorff space. Then every open neighborhood $x \in U$ admits an open neighborhood $x \in V$ with compact closure and with $\overline{V} \subset U$. Additionally, $V$ can be chosen to be regular such that $V$ and $U - \overline{V}$ are disjoint non-empty open subsets of $U$ with $U = V \uplus \partial V \uplus (U - \overline{V})$.
\end{proposition}

\section{Constructing the Desired Measure}

\begin{theorem}
Let $(X,\Sigma_X)$ be a measurable space generated by a second-countable locally compact Hausdorff non-atomic space $(X,\tau_X)$. Then there exists a finite regular non-atomic strictly-positive measure on $(X,\Sigma_X)$.
\end{theorem}\begin{proof}

First, we need a candidate for the sequence of rings of sets. Since $(X,\tau_X)$ is a second-countable locally compact regular space, there exists a countable basis of nonempty regular open sets in $X$ with compact closure. Let $(V_i) = (V_i)_{i=1}^\infty$ be one such basis expressed as a sequence. 

Regularity of the basis sequence grants several useful properties: $V_i$ being regular implies that $\overline{V_i}$ and $V_i^e$ are also regular; finite intersections of regular open sets are regular open sets; $A$ and $B$ being nonempty regular open sets with $A$ intersecting $\partial B$ implies that $A \cap B$ and $A \cap B^e$ are nonempty open sets.

The following lemma produces a sequence of rings of sets, denoted by ($\mc{D}_k$), and a limit ring of subsets $\mc{D}$ which contains basis sets  $V_i$ for $i \in \mb{N}$ and generates all Borel sets from $(X,\tau_X)$.

\begin{lemma}\label{lem_rings_of_sets}
Let $(X,\Sigma_X)$ be a second-countable non-atomic locally compact Hausdorff space with previously established basis sequence $(V_i)_{i=1}^\infty$. For each $k \in \mb{N}$, define
\begin{align*}
\mc{A}_k &:= \{\cap_{i=1}^k R_i ~|~ R_i =V_i \text{ or } R_i = V_i^e \text{ for all } 1 \le i \le k\} - \{\varnothing, \cap_{i=1}^k V_i^e\}, \\
\mc{B}_k &:= \{\uplus_{j=1}^n S_j~|~ n \in \mb{N}_0, S_j \in \mc{A}_k\}, \\
\mc{C}_k &:= \{C ~|~ C \subseteq \cup_{i=1}^k \partial V_i\}, \\
\mc{D}_k &:= \{B \uplus C ~|~ B \in \mc{B}_k ,~ C \in \mc{C}_k\} .
\end{align*}
Then: $\mc{A}_k$ consists of pairwise disjoint nonempty regular open sets; every $A \in \mc{A}_{k+1}$ with $A \subseteq \cup_{i=1}^k V_i$ has a unique $B \in \mc{A}_k$ with $A \subseteq B$; if $S,T \in \mc{B}_k$, then $S \cup T , S\cap T,\text{and } S \cap T^e \in \mc{B}_k$; $(\mc{D}_k)$ is an ascending sequence of rings of sets; $\mc{D} := \cup_{k=1}^\infty \mc{D}_k$ is a ring of sets that generates the $\sigma$-algebra $\Sigma_X$ and is insensitive to permutations on $(V_i)_{i=1}^\infty$.
\end{lemma}

\begin{proof}

If $\varnothing \not= A = \cap_{i=1}^{k+1} R_i \in \mc{A}_{k+1}$ with $A \subseteq \cup_{i=1}^k V_i$, then $B:= \cap_{i=1}^k R_i$ is the unique set in $\mc{A}_k$ containing $A$.

Let $S := \uplus_{i=1}^m S_i$ and $T := \uplus_{j=1}^n T_j$ be in $\mc{B}_k$. It is fairly straight-forward to show that $\mc{B}_k$ is closed under finite unions and intersections. We will show that $S \cap T^e \in \mc{B}_k$. If $m=n=1$, then $S,T \in \mc{A}_k$. Since sets in $\mc{A}_k$ are contained in each other's exteriors, the open set $\cup \{A \in \mc{A}_k - \{T\} \} \in \mc{B}_k$ is a subset of $T^e$. The remaining possible points in $T^e$ are either in $\cup_{i=1}^k \partial V_i$ or in $\cap_{i=1}^k V_i^e$, which are both disjoint with $S$. Therefore, $S \cap T^e = S \cap (\cup\{A \in \mc{A}_k - \{T\}\}) \in \mc{B}_k$. This elementary case, in tandem with closure of $\mc{B}_k$ under finite unions and intersections, proves the general case since
\begin{align*}
S \cap T^e &=  (\uplus_{i=1}^m S_i) \cap (\uplus_{j=1}^n T_j)^e =  (\uplus_{i=1}^m S_i) \cap (\cap_{j=1}^n T_j^e) = \uplus_{i=1}^m \cap_{j=1}^n ( S_i \cap T_j^e).
\end{align*}
It is clear that $\mc{C}_k$ is closed under finite unions.

If $D_1= B_1 \uplus C_1$ and $D_2= B_2 \uplus C_2$, where $B_1, B_2 \in \mc{B}_k$ and $C_1, C_2 \in \mc{C}_k$, then it follows that $D_1 \cup D_2 = (B_1 \cup B_2) \uplus (C_1 \cup C_2) \in \mc{D}_k$ via closure of $\mc{B}_k$ and $\mc{C}_k$ under finite unions. Furthermore, 
\begin{align*}
D_1 - D_2 &= (B_1 \uplus C_1) \cap (B_2 \uplus C_2)^c = (B_1 \uplus C_1) \cap (B_2^c \cap C_2^c) \\
&= (B_1 \cap B_2^c \cap C_2^c) \uplus (C_1 \cap B_2^c \cap C_2^c) \\
&= (B_1 \cap B_2^c) \uplus (C_1 \cap B_2^c \cap C_2^c) \\
&= (B_1 \cap B_2^e) \uplus ((B_1 \cap \partial B_2) \cup (C_1 \cap B_2^c \cap C_2^c)) \in \mc{D}_k,
\end{align*}

since $B_1 \cap B_2^e \in \mc{B}_k$ and the second set is in $\mc{C}_k$. Therefore, $\mc{D}_k$ is a ring of sets.  It follows easily that $\mc{D}= \cup_{k=1}^\infty \mc{D}_k$ is also a ring of sets. Since $\mc{D}$ contains all basis sets and since $\tau_X$ is second-countable, $\mc{D}$ generates $\Sigma_X$.

If $D_k := B_k \uplus C_k \in \mc{D}_k$ with $B_k \in \mc{B}_k$ and $C_k \in \mc{C}_k$, then notice that $(B_k \cap V_{k+1}) \cup (B_k \cap V_{k+1}^e) \in \mc{B}_{k+1}$ and that $C_k \cup (B_k \cap \partial V_{k+1}) \in \mc{C}_{k+1}$ ensure that $D_k \in \mc{D}_{k+1}$. Therefore, $\mc{D}_k \subseteq \mc{D}_{k+1}$ for all natural $k$.

Finally, let $\pi: \mb{N} \to \mb{N}$ be any permutation and let $\mc{A}_k', \mc{B}_k', \mc{C}_k', \mc{D}_k', \mc{D}'$ denote the collections discussed earlier for the permutated sequence $(V_{\pi(i)})_{i=1}^\infty$. 

Let $K \in \mc{D}$. Then $K \in \mc{D}_n$ for some natural $n$. Let $m$ be the smallest natural number such that $\{V_{\pi(1)}, V_{\pi(2)}, ... V_{\pi(m)}\} \supseteq \{V_1, V_2, \dots V_n\}$. It follows that $K \in \mc{D}_m' \subset \mc{D}'$. A similar argument suffices to show that $\mc{D}' \subseteq \mc{D}$. ~~$\blacksquare$
\end{proof}

Second, we need a finitely additive set function defined on $\mc{D}$.
Given a basis sequence $(V_i)_{i=1}^\infty$, we need to intuitively develop a sequence of set functions $\mu_m : \mc{A}_{m} \to [0,1]$ for $m \in \mb{N}$. The crucial idea is that when an open set $A \in \mc{A}_m$ intersects $\partial V_{m+1}$, regularity properties of $A$ and $V_{m+1}$ ensure that $A$ is fragmented by $\partial V_{m+1}$ into two nonempty open sets $A' := A \cap V_{m+1}$ and $A'':= A \cap V_{m+1}^e$ in $\mc{A}_{m+1}$. Hence, we can evenly divide the size $\mu_m(A)$ in half and distribute each to $A'$ and $A''$, meaning we insist that $\mu_{m+1}(A') = \mu_{m+1}(A'') = \frac{1}{2} \mu_m (A)$. Of course, we also need to insist that the next set function equals the previous set function for open sets in $\mc{A}_m$ that persist in $\mc{A}_{m+1}$. Finally, when ``new regions" are introduced, we can freely choose that size, and we shall do so in a way to cause all set functions to have maximum size output less than 1. These components correspond to lines 3-5 in the construction of $(\mu_m )_{m=1}^\infty$ in Lemma \ref{lem_set_functions}. The following figure illustrates how the set functions on $(\mc{A}_k)$ behave.

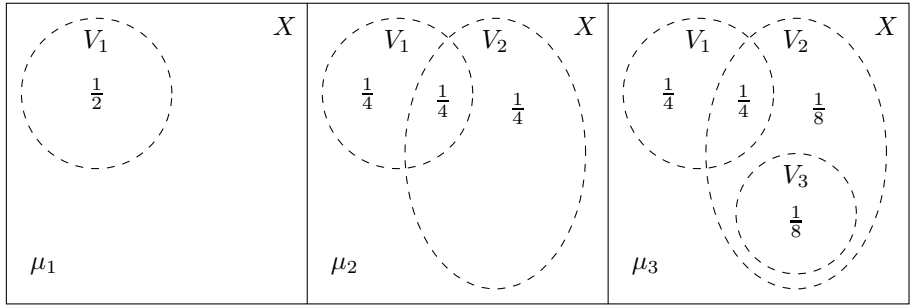
\begin{tikzpicture}
\draw (8,0)--(0,0)--(0,4)--(4,4)--(4,0);
\draw (4,4)--(12,4)--(12,0)--(8,0)--(8,4);
\draw[dashed] (1.2,2.8) ellipse (1cm and 1cm);
\draw[dashed] (5.2,2.8) ellipse (1cm and 1cm);
\draw[dashed] (9.2,2.8) ellipse (1cm and 1cm);
\draw[dashed] (6.5,2) ellipse (1.2cm and 1.8cm);
\draw[dashed] (10.5,2) ellipse (1.2cm and 1.8cm);
\draw[dashed] (10.5,1.2) ellipse (.8cm and .8cm);
\draw (.5,.5) node {$\mu_1$};
\draw (4.5,.5) node {$\mu_2$};
\draw (8.5,.5) node {$\mu_3$};
\draw (3.7,3.7) node {$X$};
\draw (7.7,3.7) node {$X$};
\draw (11.7,3.7) node {$X$};
\draw (1.2,3.5) node {$V_1$};
\draw (5.2,3.5) node {$V_1$};
\draw (9.2,3.5) node {$V_1$};
\draw (1.2,2.8) node {$\frac12$};
\draw (6.5,3.5) node {$V_2$};
\draw (10.5,3.5) node {$V_2$};
\draw (10.5,1.7) node {$V_3$};
\draw (4.8,2.8) node {$\frac14$};
\draw (8.8,2.8) node {$\frac14$};
\draw (6.8,2.6) node {$\frac14$};
\draw (5.8,2.7) node {$\frac14$};
\draw (9.8,2.7) node {$\frac14$};
\draw (10.8,2.6) node {$\frac18$};
\draw (10.5,1.1) node {$\frac18$};
\end{tikzpicture}

Once $(\mu_m)_{m=1}^\infty$ is constructed, we can easily develop a sequence of finitely additive set functions $(\kappa_n)_{n=1}^\infty$ such that $\kappa_{n+1}$ is an extension of $\kappa_n$ for all $n \in \mb{N}$, then define a finitely additive set function $\kappa$ as the overall extension of $(\kappa_n)_{n=1}^\infty$ to $\mc{D}$.

\begin{lemma}\label{lem_set_functions}
Let $(V_i)_{i=1}^\infty$ be a basis sequence for $(X,\tau_X)$, with $\mc{A}_k, \mc{B}_k, \mc{C}_k, \mc{D}_k$ and $\mc{D}$ previously developed from the basis sequence for all $k \in \mb{N}$. Recursively define a sequence of functions $\mu_m : \mc{A}_m \cup \{\varnothing\} \to [0,1]$ as follows:
\begin{align}
&\mu_1(\varnothing) := 0 ;\\
&\mu_1(V_1) := 1/2 ;\\
&\mu_m(A) := 1/2 \cdot \mu_{m-1}(B)  &\text{ if }& \varnothing \not= A \subset B \in \mc{A}_{m-1};\\
&\mu_m(A) := \mu_{m-1}(A)  &\text{ if }& A \in \mc{A}_{m-1};\\
&\mu_m(V_m - \cup_{i=1}^{m-1} \overline{V_i}) := 1/2^{m} &\text{ if }& V_m - \cup_{i=1}^{m-1} \overline{V_i} \not= \varnothing .
\end{align}

Then there exists a finitely additive set function $\kappa : \mc{D} \to [0,1]$ such that $\kappa(A) = \mu_m(A)$ when $A \in \mc{A}_m$.
\end{lemma}

\begin{proof}

Define a sequence of functions $\nu_k : \mc{B}_k \to [0,1]$ via $\nu_k(\uplus_{j=1}^n S_j) := \sum_{j=1}^n \mu_k(S_j)$, where $S_j \in \mc{A}_k$ for $1 \le j \le n$. Then define a sequence of functions $\kappa_n: \mc{D}_n \to [0,1]$ via $\kappa_n(S \cup T) = \nu_n(S)$, where $S \in \mc{B}_n$ and $T \in \mc{C}_n$. 

Finite additivity of $(\nu_k)$ and $(\kappa_n)$ is easy to verify. It follows by the definitions of $(\nu_k)$ and $(\kappa_n)$ and by (3) and (4) that any set in $\mc{D}_N$ will have the same value under all functions $\kappa_n$ with $n \ge N$. Therefore, the set function $\kappa : \mc{D} \to [0,1]$ such that $\kappa(T) = \kappa_n(T)$ when $T \in \mc{D}_n$ is well defined. Finite additivity of $\kappa$ follows from $(\kappa_n)$. $\blacksquare$
\end{proof}

Third, we need to show that $\cup_k \partial V_k$ has zero outer measure. However, notice that the set function $\kappa$ we develop \emph{depends on the order of the basis sequence} $(V_i)$. This is important, because without a careful choice made for the ordering of these basis sets, Step 4 in the outline may be difficult or impossible. What kind of sequence do we select? The next lemma serves two purposes: to provide the crucial properties needed to obtain Steps 1 and 4 in the introduction, and to help verify non-atomicity of the measure formed at the end.

\begin{lemma}\label{lem_permutation}
Let $(X,\tau_X)$ be a second-countable non-atomic locally compact Hausdorff space, let $(V_i)_{i=1}^\infty$ be a topological basis sequence, and let $(\mc{A}_k)_{k=1}^\infty$ be the sequence of collections of sets formed in Lemma \ref{lem_rings_of_sets} with respect to $(V_i)_{i=1}^\infty$. There exists a permutation $\pi: \mb{N} \to \mb{N}$ such that the set function $\kappa$ from Lemma \ref{lem_set_functions} formed with respect to $(V_{\pi(i)})_{i=1}^\infty$ has the following properties:
\begin{itemize}
\item[(a) ] The set $\cup_{k=1}^\infty \partial V_k$ has zero outer measure.
\item[(b) ] $\max\{\kappa(A) : A \in \mc{A}_n\} \to 0$ as $n \to \infty$.
\end{itemize}
\end{lemma}

\begin{proof}

Note that the boundaries $(\partial V_k)_{k=1}^\infty$ are compact, so we can find covers of them using other basis open sets. The utility of the non-atomic topological space is that we can purposefully use the closures of other basis elements to bore ``closed holes" into a given cover of a given boundary $\partial V_i$, then find a \emph{better} cover of the same boundary that does not intersect these holes. The holes should cause the new cover to have outer measure at most half of the previous cover's outer measure when $\kappa$ is formed. This process is repeated for all $\partial V_i$ countably many times so that the outer measure for each must be zero. Fortunately, the implementation below will also satisfy (b).

To form a permutation with the desired properties, we form a partition of $\mb{N}$ into three double-indexed families $\{F_{i,j} : i,j \in \mb{N} \}$, $\{G_{i,j} : i,j \in \mb{N} \}$ and $\{H_{i,j} : i,j \in \mb{N} \}$ of finite sets such that the following holds:

\begin{itemize}
\item[(a) ] $\mc{G}_{i,j} := \{V_k ~|~ k \in G_{i,j}\}$ covers $\partial V_i$ for $i,j \in \mb{N}$; that is, $\partial V_i \subset \cup \mc{G}_{i,j}$.
\item[(b) ] $\mc{F}_{i,j} := \{\overline{V_k} ~|~ k \in F_{i,j}\}$ for $i,j \in \mb{N}$ satisfies $\cup \mc{G}_{i,j} \subseteq \cup \mc{G}_{i,j-1} - \cup \mc{F}_{i,j}$ for all $i \ge 1, j \ge 2$.
\item[(c) ] With $g(i,j) := \max G_{i,j}$, for each $i\ge 1, j \ge 2$, and for each $U \in \mc{A}_{g(i+1,j-1)}$, there exists a unique $K \in \mc{F}_{i,j}$ such that $K \subset U$.
\item[(d) ] $\max G_{i,j} < \min G_{i',j'}$ when $i+j < i'+j'$ or when $i+j = i'+j'$ and $j < j'$. The same is true for $\{F_{i,j}\}_{i,j=1}^\infty$ and $\{H_{i,j}\}_{i,j=1}^\infty$ whenever the compared sets are both nonempty.
\item[(e) ]  $H_{i,j}$ are remainder sets; that is, \\
$H_{1,1} = \{1, \dots,g(1,1)\} - G_{1,1}$, \\
$H_{i,1} = \{g(1,i-1)+1,\dots,g(i,1)\} - G_{i,j}$ for all $i \ge 2$, and \\
$H_{i,j} = \{g(i-1,j+1)+1,\dots,g(i,j)\} - (G_{i,j} \cup F_{i,j})$ for all $i \ge 1, j \ge 2.$
\end{itemize}

After forming the sets above, we define the desired permutation $\pi : \mb{N} \to \mb{N}$, and hence the preferred sequence $(W_i)_{i=1}^\infty := (V_{\pi(i)})_{i=1}^\infty$ by making the arrangement $R_{i,j} := (F_{i,j}, G_{i,j}, H_{i,j})$ for each $i,j \ge 1$ and then stating that $R_{i,j} \le R_{i',j'}$ exactly when $i+j < i'+j'$ or when $i+j = i'+j'$ and $j < j'$; that is, make the arrangement $R_{1,1}, R_{2,1}, R_{1,2}, R_{3,1}, R_{2,2}, R_{1,3}, \cdots$. The details are provided below.

 $\{V_i\}_{i=1}^\infty$ covers the compact set $\partial V_1$, so there exists some finite subcover. Let $\mc{G}_{1,1}$ denote one possible choice, indexed by $G_{1,1} \subset \mb{N}$ with maximum index $g(1,1)$. Let $H_{1,1} := \{1,2,\cdots,g(1,1)\} - G_{1,1}$.

Now $\{V_i : i > g(1,1)\}$ forms a basis by Proposition \ref{prop_2}, so perform the same procedure on $\partial V_2$ using the new basis, obtaining the collection $\mc{G}_{2,1}$, indexed by $G_{2,1}$ and number $g(2,1)$. Define $H_{2,1} := \{g(1,1)+1, g(1,1)+2, \cdots, g(2,1)\} - G_{2,1}$. Next to construct is $\mc{G}_{1,2}, G_{1,2}, \text{ and } g(1,2)$. For each (non-empty) $A \in \mc{A}_{g(2,1)}$, there exists a basis set $V_{k(A)}$ such that $ \overline{V_{k(A)}} \subset A$ and $k(A) > g(2,1)$. Doing this for every $A \in \mc{A}_{g(2,1)}$, let $F_{1,2} := \{k(A) : A \in \mc{A}_{g(2,1)}\}$ and let $\mc{F}_{1,2} := \{ \overline{V_k} : k \in F_{1,2}\}$. The open set $\cup\mc{G}_{1,1} - \cup\mc{F}_{1,2}$ is the union of some covering from the basis $\{V_i : i > g(2,1) \land i \notin F_{1,2}\}$, so choose a finite subcovering $\mc{G}_{1,2}$, indexed by $G_{1,2}$ with maximum index $g(1,2)$, which covers $\partial V_{1,1}$. Now denote $H_{1,2} := \{g(2,1)+1, g(2,1) +2,\cdots,g(1,2)\} -  (G_{1,2} \cup F_{1,2})$.

Inductively, with $\mc{G}_{i,j}, G_{i,j}, g(i,j), F_{i,j}, H_{i,j}$ previously determined when $i+j \le m$, use the basis $\{V_i : i > g(1,m-1) \}$ to select a finite subcover $\mc{G}_{m,1}$ of $\partial V_m$ indexed by $G_{m,1}$ with maximum index $g(m,1)$. Construct $\mc{F}_{m-1,2}$, $F_{m-1,2}$, $\mc{G}_{m-1,2}$, $G_{m-1,2}$, $g(m-1,2)$ and $H_{m-1,2}$ based on the $(m,1)$ step similar to how $\mc{F}_{1,2}$, $F_{1,2}$, $\mc{G}_{1,2}$, $G_{1,2}$, $g(1,2)$ and $H_{1,2}$ were constructed based on the $(2,1)$ step. Repeat this again by constructing $\mc{F}_{m-2,3}$, $F_{m-2,3}$, $\mc{G}_{m-2,3}$, $G_{m-2,3}$, $g(m-2,3)$ and $H_{m-2,3}$ based on the $(m-1,2)$ step. Continue in this manner until $\mc{F}_{1,m}, F_{1,m}, \mc{G}_{1,m}, G_{1,m}, g(1,m)$ and $H_{1,m}$ are constructed. Now all appropriate numbers and collections have been found for when $i+j = m+1$.

Complete this process via induction.

Now let $\kappa$ be the finitely-additive set function created in Lemma \ref{lem_set_functions} with respect to $(V_{\pi(i)})$. It remains to verify the desired properties. For any $\partial V_i$, the sequence of covers $\{\mc{G}_{i,j}\}_{j=1}^\infty$ developed above will satisfy (for all natural $j$) the inequality 
$$\kappa (\cup \mc{G}_{i,j}) \le \kappa (\cup \mc{G}_{i,1}) \cdot \left(\frac{1}{2}\right)^{j-1} .$$
 This will be shown via induction. The inequality is obvious when $j=1$. Assume that the inequality is true for some natural $j$. Then (3) from Lemma \ref{lem_set_functions} ensures that
\begin{align*}
\kappa (\cup\mc{G}_{i,j+1}) &\le \kappa (\cup \mc{G}_{i,j} - \cup \mc{F}_{i,j+1}) = \frac{1}{2} \kappa (\cup\mc{G}_{i,j}) \le \frac{1}{2} \cdot \kappa (\cup \mc{G}_{i,1}) \cdot \left(\frac{1}{2}\right)^{j-2}.
\end{align*}

Therefore, it follows that $ \displaystyle \kappa^\ast (\partial V_i) \le \kappa (\cup \mc{G}_{i,j}) \le \kappa (\cup \mc{G}_{i,1}) \cdot \left(\frac{1}{2}\right)^{j-1}$ for all natural $j$, which implies that $\kappa^\ast (\partial V_i) = 0$. Since $i \in \mb{N}$ was arbitrary, it follows that all sets in $\mc{C}$ have outer measure zero, showing (a).

Let $m \in \mb{N}$. The basis sets $\{V_k: k \in F_{1,m}\}$ fragment each of the sets in $\mc{A}_{2,m-1}$, resulting in 
$$\max\{\kappa(A): A \in \mc{A}_{1,m}\} \le \dfrac{1}{2} \max \{\kappa(A): A \in \mc{A}_{2,m-1}\}.$$ 
Since $\max\{\kappa(A): A \in \mc{A}_n\}$ is a non-increasing function of $n$, it follows from above that $\max\{\kappa(A): A \in \mc{A}_n\} \to 0$ as $n \to \infty$.
~~$\blacksquare$
\end{proof}

Let $(\mc{A}'_k)$ and $(\mc{D}'_k)$ denote the collections of sets formed in Lemma \ref{lem_rings_of_sets} when applied to the permuted sequence $(V_{\pi(i)})$ (Recall that $\mc{D}' = \mc{D}$). Now we construct a measure on $\mc{D}$ using the steps from the introduction. 

Step 1: Construct the set function $\kappa: \mc{D} \to [0,1]$ with respect to $(V_{\pi(i)})$ via Lemmas \ref{lem_set_functions} and  \ref{lem_permutation}. Consequently, $\kappa$ is finitely additive.

It follows from finite additivity that $\kappa$ is $\sigma$-superadditive in $\mc{D}$; that is, given pairwise disjoint $(A_i)_{i=1}^\infty$ in $\mc{D}$, we have that $\kappa( \uplus_{i=1}^\infty A_i) \ge \sum_{i=1}^\infty \kappa(A_i)$. Now we need to show that $\kappa$ is $\sigma$-subadditive in $\mc{D}$. To this end, we consider the outer measure 
\begin{align*}
\kappa^\ast (A) :&= \inf\left\{\sum_{i=1}^\infty \kappa(A_i) : A \subseteq \cup_{i=1}^\infty A_i \text{ and } A_i \in \mc{D} \text{ are open for } i \in \mb{N}\right\},
\end{align*}
which is known to be finitely additive and $\sigma$-subadditive. By showing that $\kappa = \kappa^\ast$ on the ring $\mc{D}$, $\kappa$ will be a premeasure on $\mc{D}$ capable of being extended to a measure on Borel $\sigma$-algebra $\Sigma_X$, thus we proceed to Step 2.

Step 2: To show that $\kappa^\ast \le \kappa$ on open sets in $\mc{D}$, let $U \in \mc{D}$ be open. Then $U \subseteq U \cup \varnothing \cup \varnothing \cup \dots$, so we obtain that $\kappa^\ast (U) \le \kappa(U) + \kappa(\varnothing) + \kappa(\varnothing) + \dots = \kappa(U)$.

Step 3: The following argument shows that $\kappa \le \kappa^\ast$ on compact sets in $\mc{D}$.
Let $C \in \mc{D}$ be compact, and let $\eps > 0$. Choose some sequence $(A_i)_{i=1}^\infty$ of open sets in $\mc{D}$ with $C \subseteq \cup_{i=1}^\infty A_i$ and such that $\sum_{i=1}^\infty \kappa (A_i) \le \kappa^\ast(C) + \eps$. Then there exists some finite subcover, meaning there exists $n \in \mb{N}$ with $C \subseteq \cup_{i=1}^n A_i$ and there exists some $N \in \mb{N}$ with $C,A_1,\cdots,A_n \in \mc{D}'_N$. Therefore $\kappa(C) = \kappa_N(C) \le \sum_{i=1}^n \kappa_N(A_i) = \sum_{i=1}^n \kappa(A_i) \le \sum_{i=1}^\infty \kappa(A_i) \le \kappa^\ast(C) + \eps$.
With $\eps$ arbitrary, it follows that $\kappa \le \kappa^\ast$ on compact sets in $\mc{D}$.

Step 4: It has been shown in Lemma \ref{lem_permutation} that $\kappa^\ast(\cup_{k=1}^\infty \partial V_k) = 0$, which means that $\kappa^\ast$ is zero on $\mc{C}$.

Step 5: To show that $\kappa^\ast = \kappa$ on $\mc{D}$, let $A \in \mc{D}$. Note that $A$ is precompact. We see that $A^\circ$ is open in $\mc{D}$ and $\overline{A} - A^\circ \in \mc{C}$, so it follows that 
$$\kappa (A) \le \kappa\left(\overline{A}\right) \le \kappa^\ast \left(\overline{A}\right) = \kappa^\ast(A) = \kappa^\ast(A^\circ) \le \kappa (A^\circ) \le \kappa(A) .
$$  
Therefore, $\kappa^\ast = \kappa$ is a pre-measure on the ring $\mc{D}$. 

Step 6: Now we can apply Carath\'{e}odory's Extension Theorem. $\kappa^\ast$ extends to a measure on a $\sigma$-algebra containing $\sigma(\mc{D}) = \Sigma_X$, which can be restricted to a measure $\kappa^\dag$  on $\Sigma_X$. It follows that $\kappa^\dag$ is a finite strictly-positive measure on $(X,\Sigma_X)$ since each set in the basis sequence $(V_i)$ was assigned a positive measure by $\kappa$ on $\mc{D}$. Since $X$ is a Polish space, we automatically have that $\kappa^\dag$ is regular (see Theorem 8.1.12 in Cohn [1]).

To show that $\kappa^\dag$ is non-atomic, we apply the following lemma. 

\begin{lemma}\label{lem_non_atomic}
Let $\mu$ be a finite measure on $(X,\Sigma)$. Then $\mu$ is non-atomic if and only if for every $\eps > 0$, there exists a finite partition of $X$ into measurable sets with each set having $\mu$-measure less than $\eps$.
\end{lemma}

Using Lemma \ref{lem_non_atomic}, it suffices to show that for all $\eps > 0$, $X$ can be partitioned into a finite collection of measurable sets, each with $\kappa^\dag$ measure no more than $\eps$. Let $\eps>0$, and choose some integer $m$ with $\eps \ge 2^{1-m}$. Then according to Lemma \ref{lem_permutation}, all sets from $\mc{A}'_{g(1,m)} = \mc{A}_{g(1,m)}$ have $\kappa^\dag$ measure no more than $\eps$ each because complete fragmentation by closed holes has occurred at least $m-1$ times. Furthermore, $\kappa^\dag (\cup_{i=1}^\infty \partial V_i) = 0$ and $\kappa^\dag\left(\cap_{i=1}^{g(1,m)} V_i^e\right) \le 2^{-g(1,m)} \sum_{j=1}^\infty 2^{-j} \le \eps$. Since $X$ = $(\cup_{i=1}^\infty \partial V_i) \cup \left(\cap_{i=1}^{g(1,m)} V_i^e\right) \cup (\cup \mc{A}_{g(1,m)})$, $\kappa^\dag$ is non-atomic.

At last, $\kappa^\dag$ is a measure on $(X,\Sigma_X)$ with the sought properties.~~$\blacksquare$
\end{proof}

\section*{Bibliography}

[1] Cohn, D. \emph{Measure Theory}. Birkh\"{a}user (2013). 2nd Ed. p 245-246.

\end{document}